\documentclass[11pt]{article}

\usepackage{amsmath,amsthm,amssymb,bbm,mathtools}
\usepackage{verbatim}
\usepackage{thmtools, thm-restate}
\usepackage{subcaption}
\usepackage{tikz}
\usepackage{stmaryrd}

\usepackage{bbm}
\usepackage[left=2.4cm,right=2.4cm,top=2.4cm,bottom=2.4cm]{geometry}
\usepackage{float} 
\usepackage[]{color}

\newcommand{\black}{\black}

\usepackage[font=small,labelfont=bf]{caption}
\usepackage[bookmarks, colorlinks, breaklinks,
pdftitle={},pdfauthor={}]{hyperref}

\hypersetup{
	linkcolor=black,
	citecolor=black,
	filecolor=black,
	urlcolor=black}

\usepackage{cleveref}

\declaretheorem[]{theorem}
\declaretheorem[sibling=theorem]{lemma}

\declaretheorem[sibling=theorem]{conjecture}
\declaretheorem[sibling=theorem]{proposition}

\declaretheorem[numbered=yes]{remark}

\declaretheorem[numbered=yes]{question}

\numberwithin{equation}{section}

\usetikzlibrary{calc,decorations.markings}
\usetikzlibrary{decorations.pathmorphing}
\usetikzlibrary{decorations.pathreplacing}
\usetikzlibrary{patterns}
\usetikzlibrary{arrows}

\newcommand{\lam}{\lambda}

\newcommand{\eps}{\epsilon}
\newcommand{\bl}{\boldsymbol{\lambda}}

\newcommand{\blank}[1]{}

\newcommand{\cI}{\mathcal{I}}
\title{On the zeroes of hypergraph independence polynomials}
\date{\today}

\author{David Galvin\footnote{University of Notre Dame, dgalvin1@nd.edu}, Gwen McKinley\footnote{University of California, San Diego, gmckinley@ucsd.edu}, Will Perkins\footnote{Georgia Institute of Technology, math@willperkins.org}, Michail Sarantis\footnote{Carnegie Mellon University, msaranti@andrew.cmu.edu}, Prasad Tetali\footnote{Carnegie Mellon University, ptetali@cmu.edu}}

\begin{document}
	\maketitle

\begin{abstract}
We study the  locations of complex zeroes of  independence polynomials of bounded degree hypergraphs.  For graphs, this is a long-studied subject with applications to statistical physics, algorithms, and combinatorics. Results on zero-free regions for bounded-degree graphs include Shearer's result on the optimal zero-free disk, along with several recent results on other zero-free regions.   Much less is known for hypergraphs.  We make some steps towards an understanding of zero-free regions for  bounded-degree hypergaphs by proving that all  hypergraphs of maximum degree $\Delta$ have a zero-free disk almost as large as the optimal disk for graphs of maximum degree $\Delta$ established by Shearer (of radius $\sim 1/(e \Delta)$).  Up to logarithmic factors in $\Delta$ this is optimal, even for hypergraphs with all edge-sizes strictly greater than $2$.  We conjecture that for $k\ge 3$, $k$-uniform {\em linear} hypergraphs have a much larger zero-free disk of radius $\Omega(\Delta^{- \frac{1}{k-1}} )$. We establish this in the case of linear hypertrees.   
 
\end{abstract}

\section{Introduction}
\label{secIntro}

A hypergraph $G=(V,E)$ is a set of vertices $V$ along with a set of edges $E$ each of which is a  subset of $V$ of size at least $2$.  A hypergraph is $k$-uniform if all edges are of size $k$.  A $2$-uniform hypergraph is a graph.   The degree of a vertex $v \in V$ is the number of edges it appears in; in a hypergraph of maximum degree $\Delta$, each vertex appears in at most $\Delta$ edges.  An independent set in $G$ is a subset $I \subseteq V$ that contains no edge $e \in E$.  Let $\mathcal I(G)$  denote the set of all independent sets of $G$.  The independence polynomial of the hypergraph $G$ is 
\begin{equation}
\label{eqZdef}
 Z_G(\lam) = \sum_{I \in \cI(G)} \lam^{|I|} \,.
 \end{equation}
Independence polynomials for graphs arise in numerous contexts in mathematics, physics, and computer science, including in the study of the  {L}ov{\'a}sz  Local Lemma in probabilistic combinatorics~\cite{shearer1985problem,scott2005repulsive}, in the study of the hard-core lattice gas in statistical physics~\cite{baxter1980hard,van1994disagreement,galvin2004phase}, and in algorithmic problems of approximate counting and sampling~\cite{weitz2006counting,sly2010computational}.

In all of these settings, knowledge of the complex zeroes of $Z_G(\lam)$, or more precisely, knowledge of regions of $\mathbb C$ uniformly free from zeroes of $Z_G$ for some class of graphs, is crucial in understanding the  phenomena of interest.  For instance, as Shearer shows~\cite{shearer1985problem} (and Scott and Sokal expand upon~\cite{scott2005repulsive}), the largest negative zero of $Z_G(\lam)$ provides the optimal bound for the  {L}ov{\'a}sz  Local Lemma for a set of events  with a given dependency graph $G$;  the Yang--Lee theory of phase transitions states that phase transitions can only occur where complex zeroes condense on the real axis~\cite{yang1952statistical}; and regions free of complex zeros for bounded-degree graphs may determine the computational complexity of the associated approximate counting problem~\cite{harvey2018computing,peters2019conjecture,bezakova2019inapproximability}. 

For the class of graphs of maximum degree $\Delta$, much is known about optimal zero-free regions.  The optimal zero-free disk around $0$ has radius 
\begin{equation}
 \lam_s(\Delta) := \frac{(\Delta-1)^{\Delta-1}}{\Delta^{\Delta}}  \, ;
 \end{equation}
 that is, for any such $G$ and any $\lam \in \mathbb C$ with $|\lam| < \lam_s(\Delta)$, $Z_G(\lam) \ne 0$ ~\cite{shearer1985problem}.  Moreover, this result is tight: for any $\eps>0$ there is a graph of maximum degree $\Delta$  with a zero of $Z_G(\lam)$ of magnitude at most $\lam_s(\Delta)+ \eps$.  In fact, this graph can be taken to be a tree, and even more specifically, a finite-depth truncation of the infinite $\Delta$-regular tree.  There are also known zero-free regions that are not disks; for instance, those that extend beyond the optimal zero-free disk in the direction of the positive real axis~\cite{peters2019conjecture,bencs2018note,buys2021cayley,bencs2022complex}. 

In this paper we turn to the case of bounded-degree hypergraphs with hyperedges of size greater than $2$.  Our motivation is threefold.

\begin{enumerate}
\item Statistical physics.  In the language of statistical physics, the hard-core model on a graph is a model with pair interactions, while the hard-core model on a hypergraph with edges of size $>2$ has \textit{multi-body} interactions.  Multi-body interactions are relevant for a range of physical phenomena, but they are also often more difficult to analyze on both a rigorous and non-rigorous level. There is a vast body of literature on the convergence properties of the cluster expansion in models with pair interactions~\cite{groeneveld1962two,penrose1963convergence,ruelle1963correlation,fernandez2007analyticity,fernandez2007cluster,procacci2017convergence,nguyen2020convergence}, including bounds, like Shearer's, that are optimal or near-optimal.  The understanding of the convergence of the cluster expansion in models with multi-body interactions is much more limited, and bounds are typically non-effective (with respect to, say, graph degree), exclude interactions with a multi-body hard-core, or require a pairwise hard-core interaction in addition to a multi-body interaction~\cite{greenberg1971thermodynamic,moraal1976kirkwood,cassandro1981renormalization,rebenko1997convergence,procacci2000gas,rebenko2005polymer}.  See the discussion in~\cite{brydges1984short,jansen2020cluster} on   obstacles to proving convergence of the multibody cluster expansion.

\item Algorithms.  Zero-free regions of independence polynomials (and graph polynomials and partition functions more generally) are closely linked to the computational complexity of approximate counting and sampling. Barvinok~\cite{barvinok2016combinatorics}  developed an approach to approximate counting and sampling based on truncating the Taylor series (or cluster expansion) for $\log Z$; the accuracy of this truncation relies on the existence of a zero-free region for $Z$.  Refinements and applications of this method appear in, e.g.,~\cite{barvinok2016computing,barvinok2017computing,regts2018zero,patel2016deterministic, liu2019ising, HelmuthAlgorithmic2,harrow2020classical,bencs2021zero}.  On the other hand, computational complexity results for approximate counting can be proved using the existence of complex zeroes of $Z$~\cite{bezakova2019inapproximability,bezakova2021complexity,de2021zeros,buys2022lee,galanis2022complexity}.

\item Combinatorics. Beyond its application to the {L}ov{\'a}sz  Local Lemma, the cluster expansion has been used recently as a tool in asymptotic enumeration, e.g.~\cite{jenssen2020independent,balogh2021independent,jenssen2020homomorphisms, davies2021proof,jenssen2022independent}.  The high-level idea in many of these applications is to interpret defects from an extremal configuration (or a simple set of configurations) in a combinatorial problem as a statistical physics model with a pair interaction and then use the cluster expansion to estimate the partition function of this new model.  This approach is effective when typical configurations are very `ordered', with structure that resembles a well understood extremal example with sparse defects.  On the other hand, a powerful approach to asymptotic enumeration in the opposite regime, when typical configurations are unstructured, is that of Janson's inequality and its extensions~\cite{janson1988exponential,wormald1996perturbation,stark2018probability,mousset2020probability}.  In its most commonly applied form, this approach estimates the partition function of a hypergraph hard-core model: the scaled probability that a $p$-random subset of a ground set does not contain any of a specified family of subsets (the ground set is the vertex set and the family of subsets are the hyperedges of a hypergraph).  The estimate for this general problem in~\cite{mousset2020probability} takes the form of an exponential of a sum of terms defined in much the same way the terms of the cluster expansion are defined. That result, however, does  not show convergence of the infinite series defined by the cluster expansion.  If one could prove convergence, one could extend the results of~\cite{stark2018probability,mousset2020probability} and deduce additional consequences through control of certain moment generating functions (as in, e.g.,~\cite{cannon2020counting,jenssen2020independent} in the case of the graph cluster expansion).   This application is perhaps  our primary motivation for understanding zeroes of hypergraph independence polynomials, and it is the results of~\cite{mousset2020probability} that suggest Conjecture~\ref{conjLinear}.

\end{enumerate}

\subsection{Zero-freeness for hypergraph independence polynomials}

Our first main result is that \textit{all} hypergraphs satisfy a bound close to the Shearer bound for graphs.
\begin{theorem}
\label{thmZeroFreeDisk}
Let $G$ be a hypergraph of maximum degree $\Delta$.  Suppose 
\[ |\lam| \leq   \frac{\Delta^\Delta}{(\Delta+1)^{\Delta+1}}  = \lam_s(\Delta+1) \, .\]
Then $Z_G(\lam) \ne 0$.
\end{theorem}

In Theorem~\ref{thm:graph_bound} in Section~\ref{sec-c-over-Delta} we will prove a stronger statement:  in the multivariate setting in which each vertex $v$  receives its own activity $\lam_v$, $Z_G \ne 0$ when $|\lam_v| < \lam_s(\Delta+1)$ for all $v$.

The best previous bound on zero-free disks for hypergraph independence polynomials is due to Bencs, Csikv{\'a}ri, and Regts~\cite[Corollary 6]{BencsCsikvariRegts2021hypergraph} who proved zero-freeness for $|\lam |\leq 2^{-\Delta}$.

The bound in Theorem~\ref{thmZeroFreeDisk} is nearly tight apart from the possible improvement of substituting $\Delta$ for $\Delta+1$ in the bound (see Remark~\ref{rmkD1} below on one obstacle to this improvement).  The examples that show this near tightness  are  graphs (the  family of trees that prove tightness of Shearer's result), and so one might hope that in a $k$-uniform hypergraph with $k>2$ an improvement is possible.  Thanks to an example (provided to us by Wojciech Samotij), we know that in  general no polynomial improvement in $\Delta$ is possible.

\begin{proposition}
    \label{propExample}
    For each $k \ge 3$ odd, there is a family of $k$-uniform hypergraphs of maximum degree $\Delta$ with smallest root $\lam$ satisfying
    \[ |\lam | = O \left(  \frac{\log \Delta}{\Delta} \right) \,.\]
\end{proposition}
In Section~\ref{sec-WS} we give the details of this construction. 

Still one might hope that with additional conditions on the hypergraph a significant improvement might be obtained. 

A hypergraph is \textit{linear} if each pair of edges intersect in at most one vertex (a graph by definition must be linear).   We conjecture that  $k$-uniform, linear hypergraphs of bounded degree have much larger zero-free disks.
\begin{conjecture}
\label{conjLinear}
    For each $k \ge 2$, there exists a constant $C_k > 0$, so that the following is true.  If $G$ is a $k$-uniform, linear hypergraph of maximum degree $\Delta$ and if  
    \[ |\lam| \leq C_k  \Delta^{- \frac{1}{k-1}}  \,,\]
    then $Z_G(\lam ) \ne 0$. 
\end{conjecture}
The case $k=2$ is proved by Shearer's theorem; for larger $k$ the conjecture posits a polynomial improvement to the bound of Theorem~\ref{thmZeroFreeDisk}.

We can prove Conjecture~\ref{conjLinear} in a special case.  A linear hypertree is a connected, linear hypergraph with a unique path between any pair of vertices. 
We next show that linear hypertrees satisfy the bound of Conjecture~\ref{conjLinear}.
\begin{theorem}
\label{thmTrees}
    For each $k \ge 2$ the following is true.  If $G$ is a $k$-uniform, linear hypertree of maximum degree $\Delta$ and if  
    \begin{equation}
    \label{eqGraphBoundEQ}
     |\lam| \leq \left(\frac{\Delta-1}{\Delta}\right)^{\Delta-1} \left( 1- \Delta^{ - \frac{1}{k-1}}   \right)\Delta^{ - \frac{1}{k-1}}
     \end{equation}
    then $Z_G(\lam ) \ne 0$. In particular, 
   \begin{equation}
    \label{eqGraphBoundEQ2}
     |\lam| \leq \frac{\log 2}{2 k}  \Delta^{- \frac{1}{k-1}}\,,
     \end{equation}
     implies $Z_G(\lam ) \ne 0$.
\end{theorem}

The following example shows that one cannot hope for a polynomial improvement in $\Delta$ to this bound (or to the bound conjectured in Conjecture~\ref{conjLinear}). 

\begin{proposition}
    \label{propExample2}
    For each even $k \ge 4$  there is a family of $k$-uniform hypergraphs of maximum degree $\Delta$ with smallest root $\lam$ satisfying
    \[ |\lam | = O \left(  \left(\frac{\log \Delta}{\Delta}\right)^{\frac{1}{k-1}} \right) \,.\]
\end{proposition}
This lower bound is achieved by the $k$-uniform star with $\Delta$ edges. The details are in Section~\ref{secTreeConstruction}.

Finally we make a conjecture about the  \textit{zero-free locus} of independence polynomials of bounded-degree hypergraphs.   Using the notation from, e.g., \cite{buys2021cayley,de2021zeros,bencs2021limit},  let $\mathcal U_{\Delta}$ denote the maximal simply connected open set in $\mathbb C$  containing $0$ that is zero-free for  independence polynomials of all graphs of maximum-degree $\Delta$. Extending this notation, let  $\mathcal U_{\Delta,k}$ be the same for $k$-uniform hypergraphs (so $\mathcal U_{\Delta} = \mathcal U_{\Delta,2}$); and $\mathcal U_{\Delta, \ge k}$ the same for hypergraphs with edge-size at least $k$.  In particular, Theorem~\ref{thmZeroFreeDisk} shows that $\mathcal U_{\Delta,\ge 2}$ contains a disk of radius $\lam_s(\Delta+1)$.

\begin{conjecture}
\label{conjLocus}
    The zero-free locus of hypergraphs of maximum degree $\Delta$ is identical to the zero-free locus of graphs of maximum degree $\Delta$; that is, $\mathcal U_{\Delta, \ge 2} = \mathcal U_{\Delta}$.
\end{conjecture}
Conjecture~\ref{conjLocus} in particular implies that one can replace $\lam_s(\Delta+1)$ by $\lam_s(\Delta)$ in Theorem~\ref{thmZeroFreeDisk}.  The next question asks if in fact increasing $k$ strictly enlarges the zero-free locus.  

\begin{question}
    Is it true that for any $\Delta \ge 2$, $k\ge 2$, we have the strict containment
    \[ \mathcal U_{\Delta,k} \subset \mathcal U_{\Delta,k+1}  \, 
    ?\]
\end{question}

\subsection{Algorithms}
To describe the algorithmic consequences of Theorem~\ref{thmZeroFreeDisk}, we recall some basics of approximate counting and sampling.  A complex number $\hat Z$ is an $\eps$-relative approximation to a complex number $Z$ if 
\begin{equation}
    \left| \hat Z -Z  \right | \le \eps | Z| \,.
\end{equation}
An FPTAS (fully polynomial time approximation scheme) for a complex-valued graph polynomial $Z_G$ is an algorithm that, given $G$ and $\eps$, outputs an $\eps$-relative approximation to $Z_G$ and runs in time polynomial in $|V(G)|$ and $1/\eps$. 

\begin{theorem}
    \label{thmAlgorithm}
  For the class of hypergraphs $G$ of maximum degree $\Delta$ and maximum edge size $k$, there is an FPTAS for $Z_G(\lambda)$ when $| \lam | < \lam_s(\Delta+1)$.
\end{theorem}

The algorithm proceeds by truncating the cluster expansion (the Taylor series for $\log Z_G(\lam)$ around $0$); that is, using Barvinok's polynomial interpolation method. The fact that the exponential of the truncation provides a good approximation to $Z_G$ follows from the zero-freeness result of Theorem~\ref{thmZeroFreeDisk} and Barvinok's approximation lemma~\cite[Lemma 2.2.1]{barvinok2016combinatorics}.   The only additional ingredient is to show that the coefficients of the cluster expansion can be computed efficiently: we need to compute the coefficient of $\lam^r$ in time $\exp (O (r))$ where the implied constant in the $O(\cdot)$ may depend on $\Delta,k,$ and $\lam$, since we treat these as constants. For graphs, Patel and Regts~\cite{patel2016deterministic} showed how to compute these coefficients efficiently; the extension to hypergraphs was done by Liu, Sinclair, Srivastava~\cite{liu2019ising}. We give the details of the algorithm in Section~\ref{secAlgorithms}.  Note that there is no dependence on $k$ in the bound on $\lam$; the running time of the algorithm, however, grows like $n^{O(\log k +\log \Delta)}$ and so is only polynomial-time when $k$ and $\Delta$ are constants. 

Previous algorithmic work on hypergraph independent sets has primarily focused on the case $\lam=1$; that is, approximately counting the number of independent sets in a hypergraph (and sampling approximately uniformly).  The algorithmic results here generally fall into two categories: randomized approximation algorithms (yielding an FPRAS) based on Markov chain Monte Carlo and deterministic approximation algorithms (yielding an FPTAS) based on the method of correlation decay pioneered by Weitz~\cite{weitz2006counting}.  Examples of the first type of result include~\cite{bordewich2008path,hermon2019rapid,qiu2022perfect}; while the second set of results include~\cite{liu2014fptas,bezakova2019approximation}.  In~\cite{liu2014fptas}, Liu and Lu show that there is an FPTAS for counting independent sets in hypergraphs of maximum degree $5$, matching the bound for independent sets in graphs due to Weitz~\cite{weitz2006counting} (like Theorem~\ref{thmZeroFreeDisk}, this says things are no worse for hypergraphs that they are for graphs).  In~\cite{bezakova2019approximation}, Bez{\'a}kov{\'a},  Galanis,  Goldberg, and {\v{S}}tefankovi{\v{c}} study the case when $\lam=1$ and $k \ge 3$; in this case they can surpass the graph bound; e.g., giving an FPTAS for $k$-uniform hypergraphs when $\Delta \le 6$ (as opposed to $\Delta \le 5$ for graphs). Moreover they give an FPTAS when $\Delta \le k$, a  large improvement over the graph bound when $k$ is large.  Finally, there has been significant recent interest in a generalization of this counting problem, that of approximately counting the number of solutions to $k$-CNF formulas. 
 The case of independent sets in $k$-uniform hypergraphs is the special case of counting solutions to monotone $k$-CNF formulas. 
 Work on this problem includes~\cite{guo2019uniform,moitra2019approximate,feng2021fast,jain2022towards}.  It is an interesting direction to explore the connections between these results and the results and questions in the current paper.

\section{Zero-free regions for bounded-degree hypergraphs} 
\label{sec-c-over-Delta}

We begin by generalizing the independence polynomial to the multivariate case.  Let ${\bl}=(\lam_v: v \in V)$ be a collection of (possibly complex) {\it activities} on the vertices of $G$. The multivariate  independence polynomial of $G$ at $e$ is
$$
Z_{G}({\bl}) = \sum_{I \in {\mathcal I}(G)} \prod_{v \in I} \lam_v
$$
where ${\mathcal I}(G)$ is the set of all independent sets of $G$. If all $\lam_v$ have the same value, $\lam$ say, then $Z_{G}(\bl)$ is the independence  polynomial defined in Section~\ref{secIntro}. 

The main goal of this section is to prove the following zero-freeness result for the mulitivariate independence polynomial of bounded-degree hypergraph. 
\begin{theorem}\label{thm:graph_bound}
Let $G$ be a hypergraph of maximum degree $\Delta$. If $|\lam_v|\leq\frac{\Delta^\Delta}{(\Delta+1)^{\Delta+1}}$ for all $v\in V$, then
\[
|Z_{G}(\bl)| \geq \left(1-\frac{1}{\Delta+1}\right)^{|V|} > 0.
\]
\end{theorem}
Note that $(\Delta^\Delta)/(\Delta+1)^{(\Delta+1)} \sim e/\Delta$ as $\Delta \rightarrow \infty$. As discussed above, this theorem is tight apart from the possible substitution of $\Delta$ for $\Delta+1$ in the bound.

The proof of Theorem~\ref{thm:graph_bound} follows the broad outline of the proof of Shearer's theorem in~\cite{shearer1985problem,scott2005repulsive}.  While the proof for graphs involves the operation of removing vertices from a graph, the extension to hypergraphs is more involved: the operations we perform include removing vertices from a hypergraph as well as shrinking  edges. To keep track of these operations and to be explicit about edge and vertex sets, we will use the notation $Z_{V,E}(\bl)$ for $Z_G(\bl)$ where $G= (V,E)$. 

For $A\subset V$ with $A \neq \emptyset$, we define the following operations on the edge set $E$, whose utility will become apparent when we present \eqref{id:delete-v} and \eqref{id:add-partial-edge}, the basic deletion/contraction identities for $Z_{V,E}(\bl)$:
\begin{description}
\item[Deletion] $E-A$ is the set of edges that avoid $A$, that is, 
$$
E-A =\{e \in E: e \cap A =\emptyset\}.
$$
For $v \in V$ we write $E-v$ for $E-\{v\}$. 

\item[Contraction] $E/A$ is the set of edges of size at least $2$ that are created by deleting the elements of $A$ from all the edges in $E$, that is,
$$
E/A=(E-A)\cup \{e\setminus A:e\cap A\neq\emptyset,|e\setminus A|\geq 2\}.
$$
For $v \in V$ we write $E/v$ for $E/\{v\}$.
\item[Closure] $C(A)$ is the set of vertices outside $A$ that, together with some non-empty subset of $A$, form an edge in $E$; in other words, each of which forms an edge with any (nonempty) subset of $A$. In other words, it is the set of edges of size $1$ that are created by deleting the elements of $A$ from all the edges in $E$. Formally,
$$
C(A)=\{v: \text{there is } e\in E \text{ with } e\setminus A=\{v\}\}.
$$
For $v \in V$ we write $C(v)$ for $C(\{v\})$.
\item[Edge addition] $E+A$ is obtained from $E$ by including also $A$ as an edge, i.e. $E+A=E \cup \{A\}$.
\end{description}

We will use two fundamental identities relating the independence polynomial of a hypergraph to that of some smaller hypergraphs. Note that here (and throughout) we abuse notation somewhat: if $\bl$ is a set of weights indexed by a set $W$, and $W' \subseteq W$, then we write $Z_{W',E'}(\bl)$ when we actually mean $Z_{W',E'}(\bl')$, with $\bl'$ the restriction of the vector $\bl$ to the index set $W'$. 

Firstly we have that for any $v \in V$,
\begin{equation} \label{id:delete-v}
Z_{V,E}(\bl)=Z_{V\setminus\{v\},E-v}(\bl)+w_vZ_{V \setminus(\{v\}\cup C(v)),E/v}(\bl) \,.
\end{equation}
The identity follows by first considering  those independent sets in $G$ that do not contain $v$ and then those that do. 

Secondly, for all $A \subseteq V$ such that there is no $e \in E$ with $e \subseteq A$ we have
\begin{equation} \label{id:add-partial-edge}
Z_{V,E+A}(\bl)=Z_{V,E}(\bl)-Z_{V\setminus (A \cup C(A)),E/A}(\bl)\prod_{x \in A}\lam_x.
\end{equation}
This identity follow by observing that all independent sets in $(V,E)$ are independent sets in $(V,E+A)$, except those that contain all of $A$; and the extensions of $A$ to an independent set in $(V,E)$ are precisely the independent sets in $V\setminus (A \cup C(A)),E/A$. Note that if there is an $e \in E$ with $e \subseteq A$ then \eqref{id:add-partial-edge} becomes the simpler
$$
Z_{V,E+A}(\bl)=Z_{V,E}(\bl).
$$

For brevity, in what follows we will write 
\begin{itemize}
\item $G-v$ for the hypergraph $(V\setminus\{v\},E-v)$, and $G-A$ for $(V\setminus A,E-A)$
\item $G/v$ for $(V \setminus(\{v\}\cup C(v),E/v)$, 
\item $G+A$ for $(V,E+A)$ and 
\item $G/A$ for $(V\setminus (A \cup C(A)),E/A)$.
\end{itemize}
With this notation, \eqref{id:delete-v} and \eqref{id:add-partial-edge} become
\begin{equation}\label{eq:vertex_delete}
Z_G(\bl)=Z_{G-v}(\bl)+\lam_vZ_{G/v}(\bl)
\end{equation}
and
\begin{equation}\label{eq:edge_add}
Z_G(\bl)=Z_{G+A}(\bl)+Z_{G/A}(\bl)\prod_{x \in A}\lam_x.
\end{equation}	

We define a collection of {\it admissible subhypergraphs} of $G=(V,E)$ as follows: $G$ itself is admissible, and if $M=(V_M,E_M)$ is some admissible subhypergraph of $G$, then so are each of
\begin{itemize}
\item $M-v$ and $M/v$ for any $v \in V_M$,
\item $M+A$ and $M/A$ where $A\subseteq V_M$ is any proper subset of an $e \in E_G$.
\end{itemize}

We are now ready to present the proof of Theorem \ref{thm:graph_bound}.
	
\begin{proof}[Proof of Theorem~\ref{thm:graph_bound}]
The overall plan is to show that as long as the hypothesized condition on $\bl$ holds, we have that if $M$ is any non-empty admissible subhypergraph of $G$ and $v$ is any vertex in $V_M$ then 
$$
\left|\frac{Z_M(\bl)}{Z_{M-v}(\bl)}\right| \geq 1-s
$$
where $s<1$ may be chosen to be independent $M$ and $v$. Iterating this over an ordering $v_1, v_2, \ldots, v_{|V(G)|} $ of the vertices of $G$ then yields, via a telescoping product,
\begin{equation} \label{eq-tele-gen}
|Z_M(\bl)| = \prod_{i=1}^{|V(G)|} \left|\frac{Z_{G-\{v_1, \ldots, v_{i-1}\}}(\bl)}{Z_{G-\{v_1, \ldots, v_i\}}(\bl)}\right| \geq (1-s)^{|V(G)|} > 0. 
\end{equation}
To do this,  we will need to control how the independence polynomial changes in going from $M$ to $M-v$. It will turn out that in tandem with this we will also need to control how it changes in going from $M$ to $M+A$. To achieve both of these tasks, we will carry out a parallel induction. 

To state the induction hypothesis precisely, we introduce some notation. For an admissible subhypergraph $M=(V_M,E_M)$, and a non-empty subset $A \subseteq V_M$ with $|A|=a$, we define, for each positive integer $b$, the quantity
$$
n_{ab} = \#\left({S \subseteq V_M,~S \cap A = \emptyset,~|S|=b \text{ such that }  \atop S \cup T \in E_M \text{ for some non-empty } T \subseteq A}\right)
$$
when $a=1$ or $a>1$ and $b>1$. Note that if $A=\{v\}$ (so $a=1$) then $n_{ab}$ is simply the number of edges in $M$ of size $b+1$ that include $v$.

For $b=1$ and $a>1$, we slightly modify this to
$$n_{a1}=\#\left({S \subseteq V_M,~S \cap A = \emptyset,~|S|=1 \text{ such that }  \atop S \cup T \in E_M \text{ for some non-empty } T \subseteq A}\right)+|A|,$$
i.e. $n_{a1}=\#\left(2\text{-edges incident to }A\right)+|A|$.

In the notation we suppress the dependence of $n_{ab}$ on $M$ and $A$. 

We are now ready to state the main result of this section precisely. Let $R>0$ and $s_j \in (0,1)$, $j=1, 2, \ldots$ be constants that satisfy the following conditions for every admissible subhypergraph $M=(V_M,E_M)$ of $G$: 
\begin{itemize}
\item For every $A \subseteq V_M$ with $|A|=1$ (i.e., every $v \in V_M$) we have
\begin{equation} \label{R-bound-1}
R \leq s_1 \prod_{i \geq 1} (1-s_i)^{n_{1i}},
\end{equation}
\item and for every $A \subseteq M$ with $|A|=j\geq 2$ such that $M+A$ is admissible we have
\begin{equation} \label{R-bound-j}
R^j \leq s_j \prod_{i \geq 1} (1-s_i)^{n_{ji}}. 
\end{equation}
\end{itemize}
Here, finally, is the statement that we will prove. Let $\bl$ be such that $|w_x| \leq R$ for all $x \in V$. Let $M=(V_M,E_M)$ be an admissible subhypergraph of $G$. For $v \in V_M$ we have 
\begin{equation}\label{eq:vertex_induction}
|Z_M(\bl)|\geq (1-s_1)|Z_{M-v}(\bl)|~\text{({\it vertex deletion identity})}
\end{equation}
while for $A \subseteq M$ with $|A|=j\geq 2$ such that $M+A$ is admissible we have
\begin{equation}\label{eq:hyperedge_induction}
|Z_M(\bl)|\geq (1-s_j)|Z_{M+A}(\bl)|~\text{({\it edge addition identity}).}
\end{equation}

Before proving \eqref{eq:vertex_induction} and \eqref{eq:hyperedge_induction}, we show that it is enough to complete the proof of Theorem \ref{thm:graph_bound}. The key observations are that for any admissible $M$ and $v \in V_M$ we have
$$
\sum_{b \geq 1} n_{1b} = {\rm deg}_M(v) \leq \Delta
$$
and for any admissible $M$ and $A \subseteq M$ with $|A|=a\geq 2$ such that $M+A$ is admissible we have
$$
\sum_{b \geq 1} n_{ab} \leq a\Delta.
$$
(Indeed, consider $x \in A$. This is in at most $\Delta-1$ edges of $M$, since $M+A$ is admissible which means $x$ has degree at most $\Delta$ in $M+A$. Thus there are at most $\Delta-1$ $S$'s disjoint from $A$ such that $S \cup T \in E_M$ with $T \subseteq A$ and $x \in T$. Since $|A|=a$, it follows that there are at most $a(\Delta-1)$ $S$'s disjoint from $A$ such that $S \cup T \in E_M$ with $T \subseteq A$ and $T \neq \emptyset$. We now add the $a$ vertices of $|A|$ deleted to get $a\Delta$ as a bound.)

It follows that if let all $s_i$'s have a common value, $s$ say, then  \eqref{R-bound-1} and \eqref{R-bound-j} are implied by
$$
R^j \leq s  (1-s)^{j\Delta} \text{ or } R \leq s^{1/j}(1-s)^\Delta
$$
for $j=1, 2, \ldots$. Since $s \in (0,1)$, all of these are implied by $R \leq s(1-s)^\Delta$. We may now take $s=1/(\Delta +1)$, leading to  $R=\Delta^\Delta/(\Delta+1)^{(\Delta+1)}$. Theorem \ref{thm:graph_bound} now follows via the telescoping product \eqref{eq-tele-gen}.

		Now, to prove \eqref{eq:vertex_induction} and \eqref{eq:hyperedge_induction}, we will proceed by induction on $|V_M|$, showing firstly that an $|M|=n+1$ instance of \eqref{eq:vertex_induction} can be deduced from \eqref{eq:vertex_induction} and \eqref{eq:hyperedge_induction} for $|M|\leq n$, and secondly that an $|M|=n+1$ instance of \eqref{eq:hyperedge_induction} can be deduced from \eqref{eq:vertex_induction} and \eqref{eq:hyperedge_induction} for $|M|\leq n$ as well as \eqref{eq:vertex_induction} for $|M|=n+1$. 
		
		As a base case we take $M=\varnothing$, where both \eqref{eq:vertex_induction} and \eqref{eq:hyperedge_induction} are vacuously true. 
		
		For the first of the induction steps, assume that \eqref{eq:vertex_induction} and \eqref{eq:hyperedge_induction} both hold for all admissible $M'$ with $|V_{M'}|\leq n$, and let $M=(V_M,E_M)$ be an admissible subhypergraph of $G$ with $|V_M|=n+1$ vertices. Let $v$ be any vertex of $M$. To show that \eqref{eq:vertex_induction} holds, we begin by applying identity \eqref{eq:vertex_delete}:
		\[
		Z_M(\bl)=Z_{M-v}(\bl)+\lam_vZ_{M/v}(\bl).
		\]
		Via the reverse triangle inequality and the assumption $|\lam_v|\leq R$, this gives
		\begin{equation}\label{eq:triangle1}
			|Z_M(\bl)|\geq |Z_{M-v}(\bl)|-R|Z_{M/v}(\bl)|.
		\end{equation}
		To extract our desired lower bound on $|Z_M(\bl)|$ we will use the induction hypotheses (for admissible subhypergraphs with at most $n$ vertices) to bound $|Z_{M/v}(\bl)|$ in terms of $|Z_{M-v}(\bl)|$. Observe that $M/v$ may be obtained from $M-v$ by a sequence of at most $\Delta$ many vertex deletions and edge additions, and that at each step, we obtain an admissible subhypergraph of $G$ (allowing us to repeatedly apply the induction hypotheses). 
		More explicitly, beginning with $M-v$, we first delete all 2-neighbors of $v$ (that is, vertices $u$ such that $\{u,v\} \in E_{M-v}$) to obtain $M-(\{v\} \cup C(v))$, a total of $|C(v)|~(=n_{11})$ many applications of vertex deletion. By \eqref{eq:vertex_induction}, this gives 
		\begin{equation}\label{eq:induct1}
			|Z_{M-v}|\geq (1-s_1)^{n_{11}}\cdot |Z_{M-(\{v\} \cup C(v))}|.
		\end{equation}
		Then, to obtain $M/v$ from $M-(\{v\} \cup C(v))$, we add a $j$-edge $A$ for each $\{v\} \cup A \in E_M$ for all $j \geq 2$. (Since $v$ is not a vertex in $M-(\{v\} \cup C(v))$, each of these edge additions produces an admissible subhypergraph of $G$). The number of $j$-edge additions performed at this step is at most the number of $(j+1)$-edges in $M$ that include $v$, which recall we have denoted $n_{1j}$. 
		By \eqref{eq:hyperedge_induction}, this gives
		\begin{equation}\label{eq:induct2}
			|Z_{M-(\{v\}\cup C(v))}|\geq \prod_{i\geq 2}(1-s_i)^{n_{1i}}\cdot |Z_{M/v}|.
		\end{equation}
		So combining \eqref{eq:induct1} and \eqref{eq:induct2}, we obtain the following relationship between $|Z_{M/v}(\bl)|$ and $|Z_{M-v}(\bl)|$:
		\begin{equation} \label{induct3}
			|Z_{M-v}|\geq \prod_{i\geq 1}(1-s_i)^{n_{1i}}\cdot |Z_{M/v}|.
		\end{equation}
		And combining \eqref{induct3} with \eqref{eq:triangle1}, we see that
		\begin{align} \label{completing-vertex-induction} 
			|Z_M(\bl)|
			&\geq |Z_{M-v}(\bl)|\left(1-\frac{R}{\prod_{i\geq 1}(1-s_{i})^{n_{1i}}}\right).
		\end{align}
		Thus our induction hypothesis \eqref{eq:vertex_induction} is satisfied for $M$ (and $v$) as long as the inequality
		\begin{equation}\label{eq:inductioneq1}
			\left(1-\frac{R}{\prod_{i\geq 1}(1-s_{i})^{n_{1i}}}\right)\geq (1-s_1)
		\end{equation}
		is satisfied; but this is equivalent to \eqref{R-bound-1}.

		Now let us turn to the second part of the induction step, verifying \eqref{eq:hyperedge_induction}. Again, let $M$ be any admissible subhypergraph of $G$ with $n+1$ vertices, and consider any edge $\{v_1,\dots,v_{\ell-j},x_1,\dots,x_j\}$ in $G$ with $x_1,\dots,x_j\in V_M$ and $v_1,\dots,v_{\ell-j}\not\in V_M$. (So here we are taking $A=\{x_1,\dots,x_j\}$; this will be convenient as we will be dealing with the elements of $A$ one after another.) 
		First we dispense with a somewhat trivial case: if $\{x_1,\dots,x_j\}$ contains a hyperedge of $M$ 
		then \eqref{eq:hyperedge_induction} is automatically satisfied, since $Z_M(\bl) = Z_{M+\{x_1,\dots,x_j\}}(\bl)$, and so $|Z_M(\bl)| \geq (1-s_j)\, |Z_{M+\{x_1,\dots,x_j\}}(\bl)|$ for any $s_j\in(0,1)$. So from here on we may assume that no $e\in E_M$ is contained in $\{x_1,\dots,x_j\}$. 
		
		We begin by applying identity \eqref{eq:edge_add} to get
		\[
		Z_M(\bl)=Z_{M+\{x_1,\dots,x_j\}}(\bl)+\lam_{x_1}\dots \lam_{x_j}Z_{M/\{x_1,\dots,x_j\}}(\bl).
		\]
		As before, using the reverse triangle inequality and noting that $|\lam_{x_i}|\leq R$ for all $i$, we obtain
		\begin{equation}\label{eq:triangle2}
			|Z_M(\bl)|\geq |Z_{M+\{x_1,\dots,x_j\}}(\bl)|-R^j|Z_{M/\{x_1,\dots,x_j\}}(\bl)|.
		\end{equation}
		To obtain a lower bound on $|Z_M(\bl)|$, we will apply our induction hypotheses to bound $|Z_{M/\{x_1,\dots,x_j\}}(\bl)|$ in terms of $|Z_{M+\{x_1,\dots,x_j\}}(\bl)|$. Notice that $M/\{x_1,\dots,x_j\}$ may be obtained from $M+\{x_1,\dots,x_j\}$ by the following sequence of steps: Starting from $M+\{x_1,\dots,x_j\}$, we perform the vertex deletion operation $j+|C(\{x_1,\dots,x_j\})|~(=n_{j1})$ many times. Then we add (as edges) all the $i$-sets $\{a_1,\ldots, a_i\}$ such that $a_1, \ldots, a_i \in V_M$ and $\{x_{k_1},\dots, x_{k_t},a_1,\ldots, a_i\} \in E_M$ for some indices $k_1,\dots, k_t$; note that there are $n_{ji}$ such $i$-sets for each possible $i$.

		Now by \eqref{eq:vertex_induction} and \eqref{eq:hyperedge_induction} we have 
		\begin{equation} \label{inq:secondind}
			|Z_{M+\{x_1,\dots,x_j\}}|\geq \prod_{i\geq 1}(1-s_{i})^{n_{ji}}\cdot |Z_{M/\{x_1,\dots,x_j\}}|.
		\end{equation}
		(Notice that the first application of \eqref{eq:vertex_induction} in this sequence is to delete a vertex from ${M+\{x_1,\dots,x_j\}}$, which has $n+1$ vertices. This is a valid application, since we are assuming the $n+1$ case of \eqref{eq:vertex_induction} of the induction hypothesis.)

		Combining \eqref{inq:secondind} with \eqref{eq:triangle2} yields 
		\begin{align*}
			|Z_M(\bl)|
			&\geq |Z_{M+\{x_1,\dots,x_j\}}|\left(1-\frac{R^j}{\prod_{i\geq 1}(1-s_i)^{n_{ji}}}\right).
		\end{align*}
		So we obtain \eqref{eq:hyperedge_induction} for $M$ as long as 
		\begin{equation*}
			\left(1-\frac{R^j}{\prod_{i\geq 1}(1-s_i)^{n_{ji}}}\right)\geq (1-s_j),
		\end{equation*}
		which is equivalent to \eqref{R-bound-j}. This completes the induction.
			\end{proof}

\begin{remark}
\label{rmkD1}
		    There is a trick that allows one to replace $\Delta+1$ by $\Delta$ in the graph case, achieving the optimal bound: by  starting from a vertex of degree $\Delta$, every vertex edited by the subsequent application of the deletion-contraction identity has degree strictly less that $\Delta$. This is no longer true for hypergraphs. When we start from a vertex of degree $\Delta$, we might keep creating and editing vertices of total degree (or even worse, graph-degree) $\Delta$.
\end{remark}

\section{Linear hypertrees} \label{sec-genTree}

 In the special case of linear hypertrees, we  improve the bound in Theorem 8 by using the same inductive argument, but taking into account the special structure of the hypergraphs obtained during the deletion/contraction operations.  By choosing the vertex/edge appropriately at each step, we will ensure that the resulting hypergraphs have at most one small (few vertices) edge in every connected component. Since small edges are more `costly' to alter, as we will see below, this will yield the desired improvement in the final bound.  The following is the multivariate extension of Theorem~\ref{thmTrees}.

\begin{theorem}\label{thm:linearTrees_bound}
	For each $k\geq 3$, 
 the following holds.
 For any $k$-uniform, linear hypertree $G=(V,E)$ with maximum degree $\Delta$, if 
 \begin{equation}
    \label{eqGraphBoundEQthm2}
     |\lam_v| \leq \left(\frac{\Delta-1}{\Delta}\right)^{\Delta-1} \left( 1- \Delta^{ - \frac{1}{k-1}}   \right)\Delta^{ - \frac{1}{k-1}}
     \end{equation}
 for all $v\in V$, then
	\[
	|Z_G(\bl)| \geq \left(1-\Delta^{- \frac{1}{k-1}}\right)^{|V|} > 0.
	\]
 Moreover, $|\lam_v| \le \frac{\log 2}{2k} \Delta^{ - \frac{1}{k-1}}$ implies~\eqref{eqGraphBoundEQthm2}.
\end{theorem}

\begin{proof} As in the proof of \Cref{thm:graph_bound}, we will bound $|Z_G(\bl)|$ by measuring the effect that removing a single vertex can have on the independence polynomial (and iterating this procedure until we reach the empty graph). In this proof we will carefully control which vertices are removed and which subhypergraphs can be produced as a result. We will first give the precise statements \eqref{eq:genTree_vertex_induction} and \eqref{eq:genTree_edge_induction} that will be proved by induction (which bound the change in the independence polynomial after removing a vertex or adding an edge), before showing how the theorem follows, and finally establishing \eqref{eq:genTree_vertex_induction} and \eqref{eq:genTree_edge_induction}.\\

Let $M$ be any admissible subhypergraph of $G$ in which every connected component has at most one edge of size strictly less than $k$. Let $R>0$ and $s_1,\dots, s_{k-1}\in (0,1)$ be any constants satisfying 
	\begin{equation}\label{eq:RTreeBound1}
		R \leq s_1 (1-s_j)\left(1-s_{k-1}\right)^{\Delta}
	\end{equation}
	for all $2\leq j\leq k-2$, and
	\begin{equation}\label{eq:RTreeBound2}
		R^j\leq s_j(1-s_1)^{j}(1-s_{k-1})^{j(\Delta-1)}
	\end{equation}
for all $2\leq j\leq k-1$; and let $\bl$ satisfy $|\lam_v|\leq R$ for all $v\in V$. We will inductively prove the following two bounds; notice that both bounds are proved only under careful structural assumptions.

\begin{itemize}	
	\item First, let $v$ be any vertex of $M$ contained in an edge of size less than $k$, or let $v$ be any vertex in a connected component of $M$ where all edges have size $k$. Then we show that
	\begin{equation}\label{eq:genTree_vertex_induction}
		|Z_M(\bl)|\geq (1-s_1)\,|Z_{M-v}(\bl)|.
	\end{equation}
	
	\item Second, let $2\leq j \leq k-1$, and let $x_1,\dots, x_j$ be any vertices of $M$ that are (respectively) in components of $M$ where all edges have size $k$, and that are in some $k$-edge $\{v_1,\dots, v_{k-j},x_1,\dots, x_j\}$ of $G$ with $v_1,\dots, v_{k-j}\not\in V(M)$. Then we show that
	\begin{equation}\label{eq:genTree_edge_induction}
		|Z_{M}(\bl)|\geq (1-s_j)\,|Z_{M+\{x_1,\dots,x_j\}}(\bl)|.
	\end{equation}
\end{itemize}

	Before proving \eqref{eq:genTree_vertex_induction} and \eqref{eq:genTree_edge_induction}, we show how an iterative application of \eqref{eq:genTree_vertex_induction} completes the proof of \Cref{thm:linearTrees_bound}. Let $\bl$ be such that $|\lam_v|\leq R$ for all $v\in V$, where $R$ and $s_1,\dots, s_{k-1}$ are any constants satisfying \eqref{eq:RTreeBound1} and \eqref{eq:RTreeBound2} above. Taking the vertices of $G$ in any order $v_1,v_2,\dots, v_{|V|}$, we write a telescoping product:
	\begin{equation*}
    |Z_G(\bl)| = \prod_{i=1}^{|V(G)|} \left|\frac{Z_{G-\{v_1, \ldots, v_{i-1}\}}(\bl)}{Z_{G-\{v_1, \ldots, v_i\}}(\bl)}\right|.
    \end{equation*}
    We may then apply \eqref{eq:genTree_vertex_induction} with $M = Z_{G-\{v_1, \ldots, v_{i-1}\}}$ and $v=v_i$; notice that these are indeed valid choices respecting our structural constraints, since all edges in $Z_{G-\{v_1, \ldots, v_{i-1}\}}$ are of size $k$. Thus, we get the bound
    \begin{equation*}
    |Z_M(\bl)| \geq (1-s_1)^{|V|}.
    \end{equation*}
    
    To establish \Cref{thm:linearTrees_bound}, it remains only to optimize the choice of $R$. It would perhaps be difficult to precisely optimize the choices of $s_1,\dots, s_{k-1}$ to make the bounds \eqref{eq:RTreeBound1} and \eqref{eq:RTreeBound2} on $R$ as large as possible. But if we take $s_1=\cdots=s_{k-2} = \Delta^{-1/(k-1)}$ and $s_{k-1} = 1/\Delta$, we will be able to choose $R= \left(\frac{\Delta-1}{\Delta}\right)^{\Delta-1} \left( 1- \Delta^{ - \frac{1}{k-1}}   \right)\Delta^{ - \frac{1}{k-1}}$ and ensure that \eqref{eq:RTreeBound1} and \eqref{eq:RTreeBound2} are satisfied. Indeed, the constraints are now
    $$R\leq \Delta^{ - \frac{1}{k-1}}\cdot\left(\frac{\Delta-1}{\Delta}\right)^\Delta$$
    and
    $$R\leq s_j^{1/j}\cdot\frac{\Delta^{1/(k-1)}-1}{\Delta^{1/(k-1)}}\cdot\left(\frac{\Delta-1}{\Delta}\right)^{\Delta-1}$$
    for all $2\leq j\leq k$. Note that for $1<j< k$ we have $s_j^{1/j}<\Delta^{- \frac{1}{k-1}}$ and thus the strongest of these constraints is for $j=k-1$, i.e.,
    \begin{align*}
        R&\leq \Delta^{ - \frac{1}{k-1}}\cdot\frac{ \Delta^{  \frac{1}{k-1}}-1}{ \Delta^{  \frac{1}{k-1}}}\cdot\left(\frac{\Delta-1}{\Delta}\right)^{\Delta-1}  \\
        &= \left(\frac{\Delta-1}{\Delta}\right)^{\Delta-1} \left( 1- \Delta^{ - \frac{1}{k-1}}   \right)\Delta^{ - \frac{1}{k-1}}.
    \end{align*}
    It is now easy to see that the above bound is lower than the one in \eqref{eq:RTreeBound1}, so we may chose $R$ to be equal to this value. Thus if $|\lam_v|\leq R = \left(\frac{\Delta-1}{\Delta}\right)^{\Delta-1} \left( 1- \Delta^{ - \frac{1}{k-1}}   \right)\Delta^{ - \frac{1}{k-1}}$ for all $v\in V$, then 
    \begin{equation*}
    |Z_M(\bl)| \geq (1-s_1)^{|V|} = \left(1-\Delta^{- \frac{1}{k-1}}\right)^{|V|},
    \end{equation*}
    finishing the proof of \Cref{thm:linearTrees_bound}.  Note also that the factor $\left(\frac{\Delta-1}{\Delta}\right)^{\Delta-1} \left( 1- \Delta^{ - \frac{1}{k-1}}   \right)$ is increasing in $\Delta$ and $\left(\frac{\Delta-1}{\Delta}\right)^{\Delta-1} \left( 1- \Delta^{ - \frac{1}{k-1}}   \right)\cdot k$ is decreasing in $k$,  so to obtain the conclusion it suffices that $|\lam_v| \le \frac{\log 2}{2k} \Delta^{ - \frac{1}{k-1}}$, since $\lim_{k \to \infty} k(1-2^{-1/(k-1)}) = \log 2 $. \\
    
    We now return to the vertex and edge deletion bounds \eqref{eq:genTree_vertex_induction} and \eqref{eq:genTree_edge_induction}, which we will prove by induction. The base case is $M=\varnothing$, where both \eqref{eq:genTree_vertex_induction} and \eqref{eq:genTree_edge_induction} are vacuously true. 
	
	For the induction step, assume that \eqref{eq:genTree_vertex_induction} and \eqref{eq:genTree_edge_induction} hold for all $M$ with $|V(M)|\leq n$, and consider any admissible subhypergraph $\Lambda$ of $G$ with $n+1$ vertices, where $\Lambda$ satisfies the assumption that every connected component has at most one edge of size strictly less than $k$. 
	
	We first establish \eqref{eq:genTree_vertex_induction} for $\Lambda$. To that end, let $v$ be any vertex of $\Lambda$ satisfying the hypotheses of \eqref{eq:genTree_vertex_induction}. To bound $|Z_\Lambda(\bl)|$, we begin by applying identity \eqref{eq:vertex_delete}:
	\[
	Z_\Lambda(\bl)=Z_{\Lambda-v}(\bl)+\lam_vZ_{\Lambda/v}(\bl).
	\]
	And using the reverse triangle inequality and noting that $|\lam_v|\leq R$ by assumption, this gives
	\begin{equation}\label{eq:tree_triangle1}
		|Z_\Lambda(\bl)|\geq |Z_{\Lambda-v}(\bl)|-R|Z_{\Lambda/v}(\bl)|.
	\end{equation}
	To extract the desired lower bound on $|Z_\Lambda(\bl)|$, we will use our induction hypotheses to bound $|Z_{\Lambda/v}(\bl)|$ in terms of $|Z_{\Lambda-v}(\bl)|$. Observe that $\Lambda/v$ may be obtained from $\Lambda-v$ by a sequence of at most $\Delta$ many vertex deletions and $j$-edge additions (for $2\leq j\leq k-1$). But we must take care to verify that at each step, all the necessary conditions are satisfied to apply the induction hypotheses.
	
	\begin{figure}[H]
		\centering
		\includegraphics[width=1\linewidth]{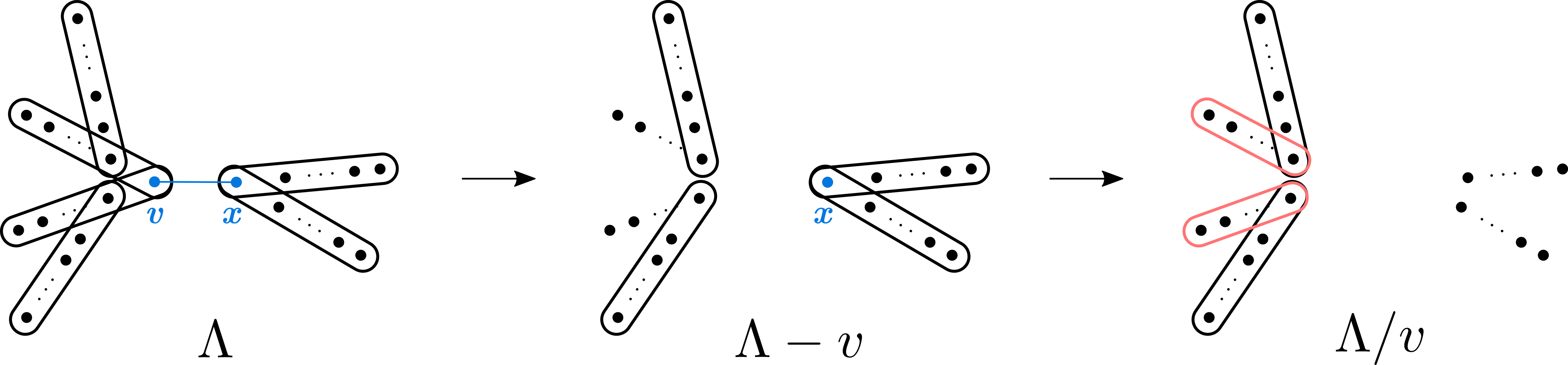}
		\caption{Obtaining $\Lambda/v$ from $\Lambda$ if $v$ is in a 2-edge}
		\label{fig:treeVertexInduct}
	\end{figure}	
	
	\begin{figure}[H]
		\centering
		\includegraphics[width=1\linewidth]{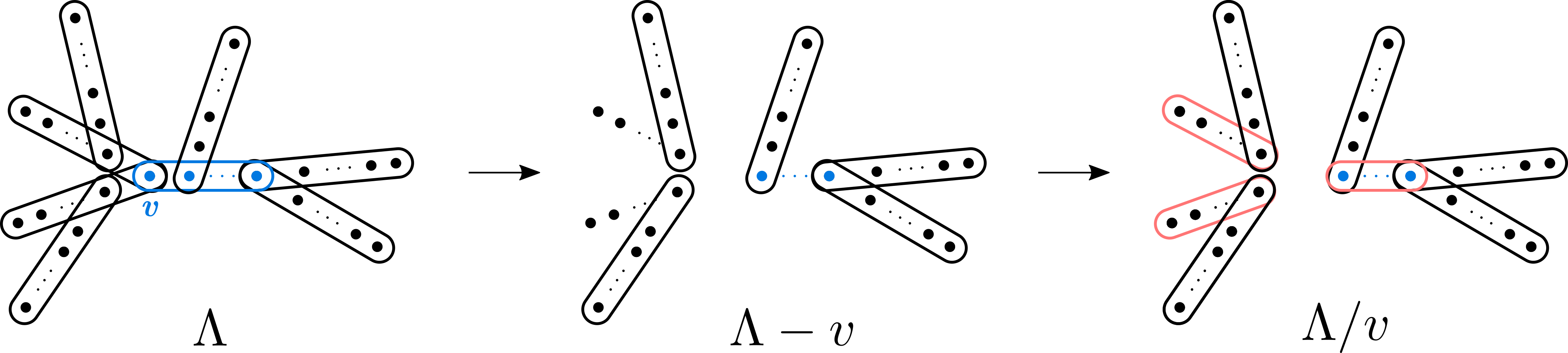}
			\caption{Obtaining $\Lambda/v$ from $\Lambda$ if $v$ is \textbf{not} in a 2-edge}
		\label{fig:treeVertexInductBig}
	\end{figure}

	First, if $v$ is contained in a 2-edge $\{v,x\}$ of $\Lambda$, then starting from $\Lambda-v$, we delete $x$. (Notice that by assumption, $v$ is contained in at most one such edge.) And we can indeed apply hypothesis \eqref{eq:genTree_vertex_induction} (with $M=\Lambda-v$), as $\Lambda-v$ has at most one edge of size less than $k$ in each connected component, and $x$ is in a component of $\Lambda-v$ where all edges have size $k$ (the edge $\{v,x\}$ is excluded from $\Lambda-v$). So applying \eqref{eq:genTree_vertex_induction}, this gives

	\begin{equation}\label{eq:tree_induct1}
		|Z_{\Lambda-v}|\geq (1-s_1)\cdot |Z_{\Lambda-v-x}|
	\end{equation}
	if $v$ has a graph neighbor $x$.

	Then, regardless of whether $v$ is contained in a 2-edge, to obtain $\Lambda/v$ from $\Lambda-v-N_2(v)$, we add a $j$-edge $\{x_1,\dots,x_j\}$ for each $\{v,x_1,\dots,x_j\}\in E(\Lambda)$, for all $2\leq j\leq k-1$.

    The number of $j$-edge additions performed at this step is at most the number of $(j+1)$-edges in $\Lambda$ adjacent to $v$, which we will denote $n_{1j}$. Note that most of the numbers $n_{1j}$ will be zero, as $v$ is adjacent to at most one edge of size less than $k$.

    And we may indeed apply induction hypothesis \eqref{eq:genTree_edge_induction} for these edge additions: regardless of the order in which we add these edges, at each step, the corresponding $x_1,\dots,x_j$ are in different components from any edges of size less than $k$ added at previous steps, since the vertex $v$ is no longer present. And at each step, the hypergraph produced has at most one edge of size less than $k$ in each connected component. So we may repeatedly apply our second induction hypothesis \eqref{eq:genTree_edge_induction} to obtain
    
    \begin{equation}\label{eq:tree_induct2}
        |Z_{\Lambda-v-N_2(v)}|\geq \prod_{i\geq 2}(1-s_i)^{n_{1i}}\cdot |Z_{\Lambda/v}|.
		\end{equation}
	Then combining \eqref{eq:tree_induct1} and \eqref{eq:tree_induct2}, we obtain the following relationship between $|Z_{\Lambda/v}(\bl)|$ and $|Z_{\Lambda-v}(\bl)|$:
		\begin{equation*} 
			|Z_{\Lambda-v}|\geq \prod_{i\geq 1}(1-s_i)^{n_{1i}}\cdot |Z_{\Lambda/v}|
		\end{equation*}
	(where $n_{1i}$ is the number of graph neighbors of $v$, which is either 0 or 1). And notice that, since $v$ is contained in at most one edge of size less than $k$, and at most $\Delta$ edges total, this bound may be simplified substantially:
		\begin{equation} \label{eq:tree_induct3}
			|Z_{\Lambda-v}|\geq  (1-s_j)\cdot (1-s_{k-1})^\Delta\cdot |Z_{\Lambda/v}|
		\end{equation}
	for some $1\leq j\leq k-2$. (Note: we may very slightly improve this bound by considering whether or not $v$ is contained in an edge of size less than $k$, but this will not substantially change our final answer. Also, we cannot control which $j$ is used; we must simply take the worst case.)

	And combining this with \eqref{eq:tree_triangle1}, we see that
	\begin{align*}
		|Z_\Lambda(\bl)|
		&\geq |Z_{\Lambda-v}(\bl)|\left(1-R\left(\frac{1}{1-s_j}\right)\left(\frac{1}{1-s_{k-1}}\right)^{\Delta}\right).
	\end{align*}
	for all $1\leq j\leq k-2$. Thus the induction hypothesis \eqref{eq:genTree_vertex_induction} is guaranteed to be satisfied for $\Lambda$ as long as 
	\begin{equation}\label{eq:tree_inductioneq1}
	1-R\left(\frac{1}{1-s_j}\right)\left(\frac{1}{1-s_{k-1}}\right)^{\Delta}\ \geq\ 1-s_1
	\end{equation}
	for each $1\leq j\leq k-2$; this is equivalent to condition \eqref{eq:RTreeBound1} above.
	
	We now proceed with the induction step for \eqref{eq:genTree_edge_induction}, which deals with edge addition. Again, we let $\Lambda$ be any admissible subhypergraph of $G$ with $n+1$ vertices satisfying the assumption that every connected component has at most one edge of size strictly less than $k$. Let $2\leq j\leq k-1$, and consider any $x_1,\dots, x_j$ in $\Lambda $ satisfying the hypotheses of \eqref{eq:genTree_edge_induction} -- that is, that $x_1,\dots, x_j$ are (respectively) in components of $\Lambda$ where all edges have size $k$, and they are in some $k$-edge $\{v_1,\dots, v_{k-j},x_1,\dots, x_j\}$ of $G$ with $v_1,\dots, v_{k-j}\not\in V(\Lambda)$. Notice that this implies $\Lambda + \{x_1,\dots,x_j\}$ is also an admissible subhypergraph of $G$, and that each component has at most one edge of size less than $k$ (the setting of our induction hypotheses).
	
	To bound $|Z_{\Lambda+\{x_1,\dots,x_j\}}|$, we begin by applying identity \eqref{eq:edge_add}:
	\[
	Z_\Lambda(\bl)=Z_{\Lambda+\{x_1,\dots,x_j\}}(\bl)+\lam_{x_1}\dots \lam_{x_j}Z_{\Lambda/\{x_1,\dots,x_j\}}(\bl).
	\]
	Then as above, by using the reverse triangle inequality and noting that $|\lam_{x_1}|,\dots, |\lam_{x_j}|\leq R$, we obtain
	\begin{equation}\label{eq:tree_triangle2}
		|Z_\Lambda(\bl)|\geq |Z_{\Lambda+\{x_1,\dots,x_j\}}(\bl)|-R^j|Z_{\Lambda/\{x_1,\dots,x_j\}}(\bl)|.
	\end{equation}
	To obtain a lower bound on $|Z_\Lambda(\bl)|$, we will apply our induction hypotheses to bound $|Z_{\Lambda/\{x_1,\dots,x_j\}}(\bl)|$ in terms of $|Z_{\Lambda+\{x_1,\dots,x_j\}}(\bl)|$. Notice that $\Lambda/\{x_1,\dots,x_j\}$ may be obtained from $\Lambda+\{x_1,\dots,x_j\}$ by a sequence of at most $j\Delta$ many operations as follows: first, we delete $x_1, x_2,\dots,$ and $x_j$. Note that, unlike in the general case, we do not need to delete $C(x_1,\dots, x_j)$; by our assumptions, this set is empty. Then add each $(k-1)$-edge $\{y_1,\dots, y_{k-1}\}$ where $\{x_i,y_1,\dots, y_{k-1}\}$ is in $E(\Lambda)$ for some $x_i$. 
	
	\begin{figure}[H]
		\centering
		\includegraphics[width=1\linewidth]{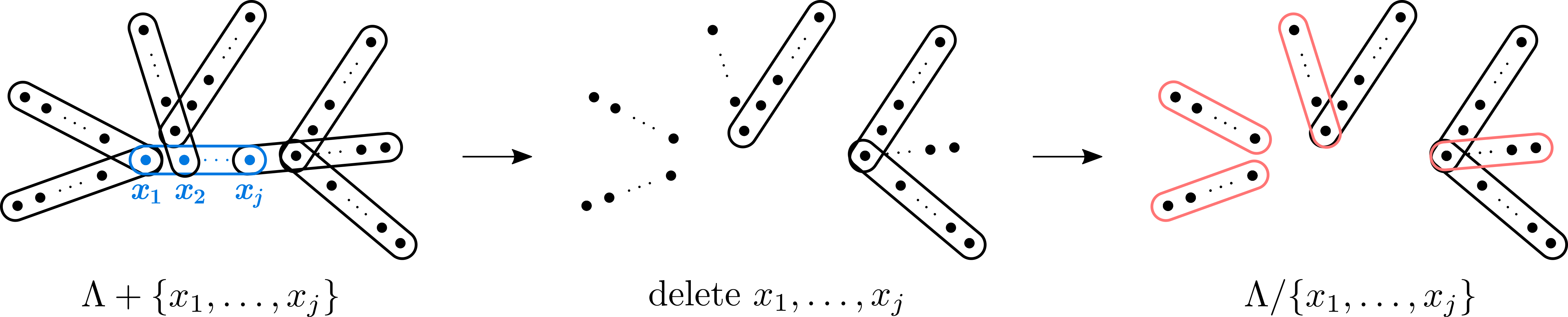}
		\caption{Obtaining $\Lambda/\{x_1,\dots,x_j\}$ from $\Lambda+\{x_1,\dots,x_j\}$}
		\label{fig:treeEdgeInduct}
	\end{figure}
	
	In total, starting from $\Lambda+\{x_1,\dots,x_j\}$, we perform the vertex deletion operation $j$ times, and the $(k-1)$-edge addition operation at most $j(\Delta-1)$ many times -- once for each of the $\leq \Delta-1$ edges adjacent to $x_1,\dots, x_j$ in $\Lambda + \{x_1,\dots,x_j\}$ respectively (excluding the edge $\{x_1,\dots,x_j\}$). And again, we must take care to verify that at each step, all the necessary conditions to apply \eqref{eq:genTree_vertex_induction} and \eqref{eq:genTree_edge_induction} are satisfied.

	First, notice that our initial application of \eqref{eq:genTree_vertex_induction} in this sequence is to delete a vertex from $\Lambda + \{x_1,\dots,x_j\}$, which has $n+1$ vertices. We are allowed to do so, since we just established the $n+1$ case of \eqref{eq:genTree_vertex_induction} above, and, without loss of generality, we begin by deleting the vertex $x_1$, which is contained in an edge of size $j<k$ in $\Lambda + \{x_1,\dots,x_j\}$. We may also delete $x_2,\dots, x_j$ next, as vertices in components where all edges have size $k$. Furthermore, at each step, regardless of the order in which we add the $(k-1)$-edges $\{y_1,\dots, y_{k-1}\}$, the corresponding $y_1,\dots, y_{k-1}$ are in different components from any $(k-1)$-edges added at previous steps, since $x_1,\dots,x_j$ are no longer present. Finally, at each step, the hypergraph produced has at most one edge of size less than $k$ in each connected component.
	
	So by \eqref{eq:genTree_vertex_induction} and \eqref{eq:genTree_edge_induction}, 
	\begin{equation*}
		|Z_{\Lambda + \{x_1,\dots,x_j\}}|\geq (1-s_1)^{j}(1-s_{k-1})^{j(\Delta-1)}\cdot |Z_{\Lambda / \{x_1,\dots,x_j\}}|.
	\end{equation*}
	
	Now, combining this inequality with \eqref{eq:tree_triangle2}, we see that 
	\begin{align*}
		|Z_\Lambda(\bl)|
		&\geq |Z_{\Lambda+\{x_1,\dots,x_j\}}|\left(1-R^ j\left(\frac{1}{1-s_1}\right)^{j}\left(\frac{1}{1-s_{k-1}}\right)^{j(\Delta-1)}\right).
	\end{align*}
	
	So we will obtain 
	\eqref{eq:genTree_edge_induction} 
	for $\Lambda$ as long as 
	\begin{equation}\label{eq:tree_inductioneq2}
		1-R^ j\left(\frac{1}{1-s_1}\right)^{j}\left(\frac{1}{1-s_{k-1}}\right)^{j(\Delta-1)}\ \geq\ 1-s_j
	\end{equation}
	for each $2\leq j\leq k-1$; this is equivalent to condition \eqref{eq:RTreeBound2} above. 
	
	Therefore, if conditions \eqref{eq:RTreeBound2} and \eqref{eq:RTreeBound2} are satisfied, this completes the induction, giving the edge and vertex deletion bounds \eqref{eq:genTree_vertex_induction} and \eqref{eq:genTree_edge_induction}, as desired.
\end{proof}

\section{Constructions}
\label{sec:Constructions}

In this section we provide the constructions that prove Propositions~\ref{propExample} and~\ref{propExample2}.

\subsection{A construction (due to Wojciech Samotij)} \label{sec-WS}
\label{secWSconstruction}

Here we describe, for each odd $k \geq 3$ and each $N$, an $k$-uniform hypergraph $H_{k,N}$ with maximum degree $\Delta \geq N$ with the following property: the univariate independence polynomial $Z_{H_{k,N}}(\lam)$ of $H_{k,N}$ is negative at $x=-k(\log \Delta)/\Delta$.  In particular, since $Z_{H_{k,N}}(\lam)\ge 1$ for $\lam \ge 0$, there must be some $\lam <0$, $|\lam|  < k(\log \Delta)/\Delta$ so that $Z_{H_{k,N}}(\lam) =0$.   This will prove Proposition~\ref{propExample}.

The vertex set of $H_{k,N}$ consists of a set $\{x_1, \ldots, x_k\}$, together with, for each $i=1, \ldots, k$, a set $\{y^i_1, \ldots, y^i_s\}$, where $s$ is chosen so that $(k-1)s \geq N$ ($s$ will later have to satisfy a further condition). For each $i=1, \ldots, k$ and each $j=1, \ldots, s$ there is an edge $(\{x_1, \ldots, x_k\}\setminus \{x_i\}) \cup y^i_j$.

In other words, we start with a set $R$ of $k$ vertices, and to each $(k-1)$-subset $R'$ of $R$ we associate a cloud $C(R')$ of $s$ vertices. An edge is formed by taking an $R'$ together with one element from its cloud $C(R')$.

Each $x_i$ is in $k-1$ $(k-1)$-subsets of $\{x_1, \ldots, x_k\}$, and each such subset can be extended to an edge of $H_{k,N}$ in $s$ ways, so the degree of each $x_i$ is $(k-1)s$. Each $y^i_j$ has degree 1. So the maximum degree of $H_{k,N}$ is $\Delta:=(k-1)s>N$.

For each $0 \leq \ell < k-1$, the total contribution to $Z_{H_{k,N}}(\lam)$ from independent sets that use exactly $\ell$ vertices from $\{x_1, \ldots, x_k\}$ is
$$
\binom{k}{\ell}\lam^\ell (1+\lam)^{sk} \,.
$$  
Since $\ell < k-1$, an arbitrary subset of the $y^i_j$ can be added to any $\ell$-subset of $\{x_1, \ldots, x_\ell\}$ without saturating an edge of  $H_{k,N}$. The total contribution to $Z_{H_{k,N}}(\lam)$ from independent sets that use exactly $k-1$ vertices from $\{x_1, \ldots, x_k\}$ is
$$
k\lam^{k-1} (1+\lam)^{s(k-1)} \,.
$$  
If $x_i$ is the one vertex from $\{x_1, \ldots, x_\ell\}$ not selected, then no $x^i_j$ can be added without saturating an edge, but any subset of the $x^{i'}_j$ for $i' \neq i$ can be. Finally, the contribution to $Z_{H_{k,N}}(\lam)$ from independent sets that use all of  $\{x_1, \ldots, x_k\}$ is simply $\lam^k$ (no extra vertices can be added without saturating an edge).

We now specialize to $\lam=-k(\log \Delta)/\Delta$. Recalling $\Delta=(k-1)s$, and using that $1+t \leq e^t$ for all real $t$, we get that for $\ell < k-1$,
$$
\left|\binom{k}{\ell}\lam^\ell (1+\lam)^{sk}\right| \leq c(k)\frac{(\log (k-1)s)^\ell}{s^{\ell+\frac{k^2}{k-1}}} 
$$
and
$$
\left|k \lam^{k-1} (1+\lam)^{s(k-1)}\right| \leq c(k)\frac{(\log (k-1)s)^{k-1}}{s^{2k-1}}\,,
$$
where $c(k)$ is a constant depending only on $k$.

On the other hand we have
$$
\lam^k = -k^k\frac{(\log (k-1)s)^k}{s^k}
$$
(note that we use here that $k$ is odd). Because $k < 2k-1$ and $k < \ell+k^2/(k-1)$ for all $\ell < k-1$, we have that for all sufficiently large $s=s(k)$,
$$
k^k\frac{(\log (k-1)s)^k}{s^k} > c(k)\left(\frac{(\log (k-1)s)^{k-1}}{s^{2k-1}} + \sum_{\ell=0}^{k-2} \frac{(\log (k-1)s)^\ell}{s^{\ell+\frac{k^2}{k-1}}}\right)
$$
and so $Z_{H_{k,N}}(\lam)<0$ at $x=-k(\log \Delta)/\Delta$, as claimed.

\subsection{A hypertree construction}
\label{secTreeConstruction}

We now give a construction to prove Proposition~\ref{propExample2} and show that the bound in Theorem~\ref{thm:linearTrees_bound} cannot be improved beyond a polylogarithmic factor in $\Delta$.\\
Consider the $k$-uniform star  of size $\Delta$, $S^k_\Delta$, which consists of $\Delta$ edges (each of size $k$) that share a single vertex,  so $S^k_\Delta$ has $1+(k-1)\Delta$ vertices in total. 
The independence polynomial of $S^k_\Delta$ is
$$Z_{S^k_\Delta}(\lambda)=(1+\lambda)^{(k-1)\Delta}+\lambda((1+\lambda)^{k-1}-\lambda^{k-1})^\Delta.$$
Let $k$ be even. We will prove that $Z_{S^k_\Delta}\left(-C\left(\frac{\log \Delta}{\Delta}\right)^{1/(k-1)}\right)<0$ for a constant $C$ and $\Delta$ large enough, and thus $Z_S$ will have a root of magnitude at most $C\left(\frac{\log \Delta}{\Delta}\right)^{1/(k-1)}$.\\
Note that the $\lambda$ we choose will clearly satisfy $|\lambda|<1$ so it is equivalent to show
$$1+\lambda\left(1-\left(\frac{\lambda}{1+\lambda}\right)^{k-1}\right)^{\Delta}<0.$$
Let $\frac{\lambda}{1+\lambda}=-f(\Delta)$, so that $\lambda= -\frac{f(\Delta)}{1+f(\Delta)}.$ When $k$ is even, we can rewrite the expression as
$$1+\lambda\left(1+f(\Delta)^{k-1}\right)^{\Delta}.$$
Set $f(\Delta):=(\log \Delta/\Delta)^{1/(k-1)}$, and then the asymptotic behavior of the above expression is
$$1+\lambda(1+\log \Delta/\Delta)^\Delta\sim1-\frac{(\log \Delta/\Delta)^{1/(k-1)}}{1+(\log \Delta/\Delta)^{1/(k-1)}}\cdot \Delta\rightarrow-\infty.$$
Note that for this $f(\Delta)$, we have $|\lambda|=\Theta\left(|f(\Delta)|\right)=\Theta\left(\left(\frac{\log \Delta}{\Delta}\right)^{1/(k-1)}\right).$

\section{Algorithms}
\label{secAlgorithms}

Given the zero-freeness result of Theorem~\ref{thmZeroFreeDisk}, we can obtain an FPTAS for $Z_G(\lam)$ and prove Theorem~\ref{thmAlgorithm} following Barvinok's method of polynomial interpolation: truncating the Taylor series for $\log Z_G(\lam)$ (in fact, the cluster expansion) around $0$ after a given number of terms.  This approach has been used in several recent works on approximate counting, including~\cite{patel2016deterministic,peters2019conjecture,HelmuthAlgorithmic2,bencs2018note,bezakova2019inapproximability,de2021zeros} on approximating the independence polynomial of bounded-degree graphs for (possibly complex) values of $\lam$.

Restating Theorem~\ref{thmAlgorithm}, our goal is to prove the following.

\begin{theorem}
    \label{thmAlgorithm2}
  For the class of hypergraphs $G$ of maximum degree $\Delta$ and maximum edge size $k$, and for complex $\lam$ satisfying $| \lam | < \lam_s(\Delta+1)$,  there is an algorithm running in time $(n/\eps)^{O_{k,\Delta}(1)}$ that computes an $\eps$-relative approximation to $Z_G(\lambda)$.
\end{theorem}

Given a polynomial $Z(\lam)$ with $Z(0) =1$, let $T_r(\lam)$ be the order-$r$ truncation of the Taylor series for $\log Z(\lam)$ around $0$.  That is,
\[  T_r(\lam)  = \sum_{j=1}^r  \frac{\lam^j}{j!} \frac{\partial^j \log Z_G(\lam) }{\partial \lam^j}  \,.\]

The connection between zero-freeness and approximation is provided by the following elementary but powerful lemma of Barvinok. 

\begin{lemma}[Barvinok~\cite{barvinok2016combinatorics}]
Let $Z(\lam)$ be a polynomial of degree at most $N$, and suppose that $Z(\lam) \ne 0$ when $|\lam| \le B$.  Then for $|\lam|< B$, 
\begin{equation}
    \left|T_r(\lam) - \log Z(\lam)  \right |  \le \frac{N ( |\lam|/B)^{r+1}}{ (r+1) (1- |\lam|/B)  } \,.
\end{equation}
\end{lemma}

We now prove Theorem~\ref{thmAlgorithm2}.

\begin{proof}
Since $Z_G(\lam)$ is a polynomial of degree at most $n = |V(G)|$, if $Z_G(\lam)$ is zero-free in a disk of radius $B$ around $0$ and  $|\lam| \le (1-\delta)B$ then $\exp(T_r(\lam))$ gives an $\eps$-relative approximation to $Z_G(\lam)$ when $r \ge C \log (n/\eps)$, where $C$ is a constant that depends only on $\delta$.  

Thus to prove Theorem~\ref{thmAlgorithm2} given Theorem~\ref{thmZeroFreeDisk} we are left with the task of computing $T_r(\lam)$ for $r = \Theta( \log (n/\eps))$ in time polynomial in $n/\eps$. Generalizing the approach of Patel and Regts~
\cite{patel2016deterministic} to hypergraphs, Liu, Sinclair, and Srivastava~\cite{liu2019ising} gave an algorithm to compute the first $r$ coefficients of the partition function of a $2$-spin model on a bounded-degree hypergraph.  Since the coefficients of the Taylor series for $\log Z$ are related to the coefficients of $Z$  through a triangular system of linear equations, this yields an algorithm to compute $T_r(\lam)$. 
Specializing their result to the hypergraph independence polynomial yields the following, which finishes the proof of Theorem~\ref{thmAlgorithm2}.

\begin{lemma}[Liu, Sinclair, Srivastava~\cite{liu2019ising}]
    Fix $k$, $\Delta$, and $C>0$. Then there is an algorithm running in time polynomial in $n/\eps$ that computes $T_r(\lam)$ for any hypergraph $G$ of maximum degree $\Delta$ and maximum edge size $k$ on $n$ vertices, where $r = \lceil C \log ( n/\eps ) \rceil$.
\end{lemma}

\end{proof}

\section*{Acknowledgements}

This work was done as part of an AIM SQuaRE workshop on `The Independence Polynomial of Hypergraphs'.  We thank AIM and their staff for their support, and we thank Tyler Helmuth for very helpful preliminary conversations.  We thank Wojciech Samotij for illuminating conversations and for providing the example in Proposition~\ref{propExample}. We also thank Weiming Feng for carefully reading the manuscript and making valuable observations that improved the presentation of Theorem~\ref{thm:graph_bound}. DG is supported in part by the Simons Foundation. WP is supported in part by the NSF grant DMS-1847451.  MS is supported in part by the Onassis Foundation - Scholarship F ZP 051-1/2019-2020. PT supported in part by the NSF grant DMS-2151283 and Alexander M. Knaster Professorship.

\end{document}